\documentclass[10pt]{amsart}
\usepackage{amsmath,amssymb,latexsym,hyperref,url}

\newtheorem{theorem}{Theorem}
\newtheorem{corollary}[theorem]{Corollary}
\newtheorem{lemma}[theorem]{Lemma}
\newtheorem*{theorem*}{Theorem}
\newtheorem*{definition}{Definition}

\theoremstyle{definition}
\newtheorem*{example}{Example}

\newcommand{\Aut}{\operatorname*{Aut}}
\newcommand{\Inn}{\operatorname*{Inn}}

\begin{document}

\title{Growing words in \\ the free group on two generators}

\author{Bobbe Cooper}
\address{
	School of Mathematics \\
	206 Church St. S.E. \\
	Minneapolis, MN 55455, USA
}

\author{Eric Rowland}
\address{
	LaCIM \\
	Universit\'e du Qu\'ebec \`a Montr\'eal \\
	Montr\'eal, QC H2X 3Y7, Canada
}



\thanks{
We thank Dennis Garity for many comments and suggestions throughout the development of this paper.
We thank the referee for a careful reading and several good suggestions.
}
\thanks{
This work was supported in part by NSF grant DMS 0139678.
}

\begin{abstract}
This paper is concerned with minimal-length representatives of equivalence classes of words in $F_2$ under $\Aut F_2$.
We give a simple inequality characterizing words of minimal length in their equivalence class.
We consider an operation that ``grows'' words from other words, increasing the length, and we study \emph{root words} --- minimal words that cannot be grown from other minimal words.
Root words are ``as minimal as possible'' in the sense that their characterization is the boundary case of the minimality inequality.
The property of being a root word is respected by equivalence classes, and the length of each root word is divisible by $4$.
\end{abstract}

\maketitle
\markboth{Cooper and Rowland}{Growing words in the free group on two generators}

\section{Introduction}\label{introduction}

In 1936 J. H. C. Whitehead~\cite{Whitehead 1936a, Whitehead 1936b} exhibited a finite set of \emph{Whitehead automorphisms} with the property that if two elements $w$ and $v$ in the free group $F_n$ are equivalent under an automorphism of $F_n$ and $v$ is of minimal length in its equivalence class, then $v$ is the image of $w$ under a sequence of Whitehead automorphisms.
Furthermore, the words obtained after applying each automorphism in this sequence are strictly decreasing in length until the minimal length is attained, after which the automorphisms leave the length fixed.
Because the set of Whitehead automorphisms is finite, this theorem provides an algorithmic foundation for determining equivalence of elements in $F_n$.

In this paper we use the Whitehead automorphisms to study minimal words in $F_2$.
We begin in this section by introducing our notation and defining the two types of Whitehead automorphisms.
Section~\ref{minimal words} uses subword counts to establish a simple characterization of minimal words in $F_2$.
In Section~\ref{root words} we introduce the notion of a root word.
Root words are ``as minimal as possible'' in a sense that we make precise.
We prove a characterization of root words analogous to our characterization of minimal words, and this allows us to derive several properties of root words and their equivalence classes.

Although the results of the paper concern elements in $F_2$, we establish notation for $F_n$ in order to state Whitehead's theorem in generality.
Throughout the paper we will use $a$ and $b$ as the generators of $F_2$, and we will denote the inverse of a generator $x$ by $\overline{x}$.
The set of \emph{letters} --- the generators and their inverses --- is referred to as the alphabet $L_n$.
Thus $L_2 = \{a,b,\overline{a},\overline{b}\}$.
We will use $x$ and $y$ as variables over $L_n$, where usually $x \not \in \{y, \overline{y}\}$.

A \emph{word} is a finite sequence of letters.
We identify each element $w \in F_n$ with its word on $L_n$ with no adjacent inverses, so that $F_n$ consists of all words $w_1 \cdots w_\ell$ such that $w_i \neq \overline{w_{i+1}}$ for $1 \leq i \leq \ell - 1$.
The length $\ell$ of $w$ is denoted $|w|$.

We are interested in equivalence classes of words in $F_n$ under the automorphism group $\Aut F_n$.
We write $w \sim v$ if $\phi(w) = v$ for some automorphism $\phi \in \Aut F_n$.

\begin{example}
The set
\begin{align*}
\{
&
aa,
bb,
\overline{aa},
\overline{bb},
abab,
abb\overline{a},
a\overline{b}a\overline{b},
a\overline{bb}\overline{a},
baa\overline{b},
baba,
b\overline{a}b\overline{a},
b\overline{aa}\overline{b},
\overline{a}bba,
\\
&
\overline{a}b\overline{a}b,
\overline{a}\overline{b}\overline{a}\overline{b},
\overline{a}\overline{bb}a,
\overline{b}aab,
\overline{b}a\overline{b}a,
\overline{b}\overline{aa}b,
\overline{b}\overline{a}\overline{b}\overline{a},
aabaab,
aabab\overline{a},
\dots
\}
\end{align*}
is an equivalence class of words in $F_2$ under $\Aut F_2$.
\end{example}

For our purposes it frequently suffices to divide $\Aut F_n$ by the subgroup $\Inn F_n$ of inner automorphisms.
Let $C_n$ be the set of words $w = w_1 \cdots w_\ell \in F_n$ such that $w_\ell \neq \overline{w_1}$.
Each word in $C_n$ is a minimal-length representative of an equivalence class of $F_n$ modulo $\Inn F_n$.
We refer to such an equivalence class as a \emph{cyclic word}.
Most cyclic words have several minimal-length representatives.
For example, the words $abab$ and $baba$ are representatives of the same cyclic word.

Whitehead's theorem distinguishes two types of automorphisms.
A \emph{Type I automorphism} is an automorphism $\phi \in \Aut F_n$ which permutes $L_n$.
Note that a permutation $\phi$ of $L_n$ can be extended to a Type I automorphism if $\overline{\phi(y)} = \phi(\overline{y})$ for all $y \in L_n$.
We will refer to Type I automorphisms as \emph{permutations}.

Let $x \in L_n$ and $A \subset L_n \backslash \{x, \overline{x}\}$.
Define a map $\phi : L_n \to F_n$ by
\[
	\phi(y) = \overline{x}^{\chi(\overline{y} \in A)} y x^{\chi(y \in A)},
\]
where $\chi(\textsf{true}) = 1$ and $\chi(\textsf{false}) = 0$.
One checks that $\phi(y)^{-1} = \phi(\overline{y})$ and that this map extends to an automorphism $\phi \in \Aut F_2$.
We write $\phi = (A,x)$ and call $\phi$ a \emph{Type II automorphism}.

A \emph{one-letter automorphism} is a Type II automorphism $\phi=(A,x)$ where the set $A$ contains only one element.
The automorphism $(\{y\},x)$ maps $x \mapsto x$, $\overline{x} \mapsto \overline{x}$, $y \mapsto yx$, and $\overline{y} \mapsto \overline{x}\overline{y}$.
We will see in Section~\ref{minimal words} that in a sense one-letter automorphisms suffice for $F_2$.

\begin{example}
The one-letter automorphism $(\{a\},b)$ maps $a \mapsto ab$ and $\overline{a} \mapsto \overline{b}\overline{a}$ and leaves $b,\overline{b}$ fixed.
\end{example}

\begin{definition}
A word $w \in F_n$ is \emph{minimal} if $|w| \leq |\phi(w)|$ for all $\phi \in \Aut F_n$.
\end{definition}

\begin{example}
Let $\phi = (\{a\},b) \in \Aut F_2$.
We have $\phi(b\overline{a}b\overline{a}) = b\overline{b}\overline{a}b\overline{b}\overline{a} = \overline{aa}$.
Therefore the word $b\overline{a}b\overline{a}$ is not minimal.
On the other hand, $\overline{aa}$ is minimal since it is not equivalent to the empty word or to any of the four words of length $1$.
\end{example}

If $w \in F_n$ and $w$ is minimal then in fact $w \in C_n$, since there is an inner automorphism that has the effect of moving the final letter of $w$ to the front.

We now state Whitehead's theorem formally.
Whitehead's proof is topological.
Rapaport~\cite{Rapaport} later provided an algebraic proof, which was simplified by Higgins and Lyndon~\cite{Higgins--Lyndon}.

\begin{theorem*}[Whitehead]
If $w,v \in F_n$ such that $w \sim v$ and $v$ is minimal, then there exists a sequence $\phi_1, \phi_2, \dots, \phi_m$ of Type I and Type II automorphisms such that
\begin{itemize}

\item
$\phi_m \cdots \phi_2 \phi_1(w) = v$ and

\item
for $0 \leq k \leq m - 1$, $|\phi_{k+1} \phi_k \cdots \phi_2 \phi_1(w)| \leq |\phi_k \cdots \phi_2 \phi_1(w)|$, with strict inequality unless $\phi_k \cdots \phi_2 \phi_1(w)$ is minimal.

\end{itemize}
\end{theorem*}

Note that since permutations do not affect the length of a word, the automorphisms $\phi_1, \phi_2, \dots, \phi_{m-1}$ can be taken to be Type II automorphisms.

The theorem immediately suggests several algorithms.
To determine whether $w \in F_n$ is minimal, it suffices to individually apply all Type II automorphisms to $w$; then $w$ is minimal if and only if the image under each automorphism is not shorter than $w$.
If $w$ is not minimal, then one may find the set of minimal words equivalent to $w$ by repeating this process; for each image that is shorter than $w$, apply all Type II automorphisms to it, keeping the images that are shorter still, until no shorter words are produced.
To determine whether two words $w$ and $v$ are equivalent, one may simply compute the set of minimal words equivalent to each and check whether they are the same.
Finally, given $\ell \geq 0$, one may compute all equivalence classes of $F_n$ containing a minimal word of length $\ell$ by first identifying which length-$\ell$ words on the alphabet $L_n$ are minimal, and then classifying them into equivalence classes.

\section{Minimal words}\label{minimal words}

A central theme of the area is that information about the equivalence class of a word can be obtained from statistics of its (contiguous) subwords.
This idea has been successfully employed in a number of recent papers on equivalence class size.
Myasnikov and Shpilrain~\cite{Myasnikov--Shpilrain} characterized the possible orbit sizes of $w \in F_n$ under an automorphism $\phi$, and for $F_2$ they showed that the number of minimal words equivalent to a minimal word $w$ is bounded by a polynomial in $|w|$.
Lee showed that the same is true for minimal $w \in F_n$ under a local condition on $w$~\cite{Lee 2006a} and determined the degree of this polynomial~\cite{Lee 2006b}.
These results imply upper bounds on the time required to determine whether two words $w, v \in F_n$ are equivalent.
For $F_2$, Khan~\cite{Khan} analyzed the structure of equivalence classes and showed that the running time is at most quadratic in $\max(|w|,|v|)$.
All these results make use of counting various subwords of $w$.

In this section we give a new characterization of minimal words in $F_2$.
Our characterization also depends on subword counts.
If $w = w_1 \cdots w_\ell$ and $v$ are nonempty words on $L_2$ such that $k = |v| \leq |w| = \ell$, let $[v]_w$ denote the number of (possibly overlapping) occurrences of the subword $v$ in $w_1 \cdots w_\ell w_1 \cdots w_{k-1}$.
If $|v| > |w|$, let $[v]_w = 0$.
Note that if $w$ and $w'$ are two minimal-length representatives of a cyclic word then $[v]_w = [v]_{w'}$.
Let $(v)_w = [v]_w + [v^{-1}]_w$ denote the total number of occurrences of $v$ and $v^{-1}$.

As we will see shortly, length-$2$ subwords are particularly important, since they track the effects of one-letter automorphisms.

\begin{example}
Let $w = aa\overline{bb}\overline{a}ba\overline{b}a$.
The length-$2$ subword counts are $(aa)_w = 2$, $(bb)_w = 1$, $(ab)_w = 1 = (ba)_w$, and $(a\overline{b})_w = 2 = (\overline{b}a)_w$.
\end{example}

\begin{theorem}\label{practicality}
If $w \in C_2$ and $x,y \in L_2$, then $(xy)_w = (yx)_w$.
\end{theorem}

\begin{proof}
We may assume $x \notin \{y,\overline{y}\}$ since otherwise the theorem is trivial.
Each occurrence of the subword $xy$ or $\overline{x}y$ in $w$ is followed by either $yx$ or $y\overline{x}$ (with some intermediate $y^\ell$).
Similarly, each occurrence of the subword $y\overline{x}$ or $\overline{yx}$ in $w$ is followed by $\overline{x}y$ or $\overline{xy}$ (with some intermediate $\overline{x}^\ell$).
Therefore
\begin{align*}
	[xy]_w + [\overline{x}y]_w &= [yx]_w + [y\overline{x}]_w, \\
	[y\overline{x}]_w + [\overline{yx}]_w &= [\overline{x}y]_w + [\overline{xy}]_w.
\end{align*}
By definition, $[xy]_w + [\overline{yx}]_w = (xy)_w$ and $[yx]_w + [\overline{xy}]_w = (yx)_w$.
Add the equations above and use these relations to obtain $(xy)_w = (yx)_w$.
\end{proof}

Note that the Type II automorphism $(\{y,\overline{y}\},x)$ with $x \notin \{y,\overline{y}\}$ is an inner automorphism on $F_2$, since it conjugates $x$ by $y$ and also (trivially) conjugates $y$ by $y$.
Consequently, automorphisms of this form do not decrease the length of any word in $C_2$.
The Type II automorphism $(\{\},x)$ is the identity map and thus does not change the length of a word in $C_2$.
Therefore by Whitehead's theorem it suffices to consider one-letter automorphisms when determining the minimality of a word in $C_2$.
This allows us to prove the following, a special case of a theorem of Rapaport~\cite[Theorem~7]{Rapaport}.
A related result was given by Khan~\cite[Corollary 2.21]{Khan}.

\begin{lemma}\label{Rapaport}
Let $w \in C_2$.
Then $w$ is minimal if and only if for all $x,y \in L_2$ with $x \notin \{y,\overline{y}\}$ we have $(y\overline{x})_w \leq (yx)_w + (yy)_w$.
\end{lemma}

\begin{proof}
The automorphism $(\{y\},x)$ causes cancellations in $w$ only between the two letters of the subwords $y\overline{x}$ and $x\overline{y}$.
The total number of cancellations is therefore $(y\overline{x})_w$.
Similarly $(\{y\},x)$ causes additions to $w$ only between the two letters of the subwords $yx$, $\overline{xy}$, $yy$, and $\overline{yy}$, totaling $(yx)_w + (yy)_w$.
Thus, the left side of the inequality is the number of cancellations in $w$ under the automorphism $(\{y\},x)$, while the right side is the number of additions.
The word $w$ is minimal if and only if the inequality holds for each choice of $x, y$.
\end{proof}

There are eight one-letter automorphisms $(\{y\},x)$ with $x \notin \{y,\overline{y}\}$.
Each of these can be written as the product $(\{y\},x) = (\{y,\overline{y}\},x) (\{ \overline{y}\},\overline{x})$ of an inner automorphism and another one-letter automorphism, so after reducing modulo $\Inn F_2$ there are four distinct one-letter automorphisms.
Therefore we have the following.

\begin{lemma}\label{one-letter automorphisms}
Let $w \in C_2$.
Then $w$ is minimal if and only if none of the four one-letter automorphisms $(\{a\},b)$, $(\{a\},\overline{b})$, $(\{b\},a)$, and $(\{b\},\overline{a})$ decrease the length of $w$.
More generally, $w$ is minimal if and only if none of the automorphisms $(\{y\},x)$, $(\{y\},\overline{x})$, $(\{x\},y)$, and $(\{x\},\overline{y})$ decrease the length of $w$ for some $x \notin \{y,\overline{y}\}$.
\end{lemma}

Lemma~\ref{one-letter automorphisms} provides the following characterization of minimal words.
Sanchez~\cite{Sanchez} also gave a characterization, although in our opinion the following is simpler to state.

\begin{theorem}\label{minimality}
Let $w \in C_2$.
The following are equivalent.
\begin{enumerate}
\item
$w$ is minimal.
\item
$|(ab)_w - (a\overline{b})_w| \leq \min((aa)_w, (bb)_w)$.
\item
$|(yx)_w - (y\overline{x})_w| \leq \min((xx)_w, (yy)_w)$ for some $x, y \in L_2$ with $x \notin \{y,\overline{y}\}$.
\end{enumerate}
\end{theorem}

\begin{proof}
There are eight inequalities to be considered in Lemma~\ref{Rapaport} over $L_2$; applying Lemma~\ref{one-letter automorphisms} to these reduces the number of distinct inequalities to four, giving that $w$ is minimal if and only if the inequalities
\begin{align*}
	(ab)_w &\leq (a\overline{b})_w + (aa)_w \\
	(a\overline{b})_w &\leq (ab)_w + (aa)_w \\
	(ba)_w &\leq (b\overline{a})_w + (bb)_w \\
	(b\overline{a})_w &\leq (ba)_w + (bb)_w
\end{align*}
hold.
This set of inequalities is equivalent to the two inequalities
\begin{align*}
	|(ab)_w - (a\overline{b})_w| &\leq (aa)_w \\
	|(ba)_w - (b\overline{a})_w| &\leq (bb)_w,
\end{align*}
and the result follows from Theorem~\ref{practicality}.
The same argument goes through with $a, b$ replaced by any $x, y$.
\end{proof}

\begin{corollary}
Let $w$ be a word on the alphabet $\{a, b\}$.
Then $w$ is minimal if and only if $[ab]_w \leq \min([aa]_w, [bb]_w)$.
\end{corollary}

\begin{corollary}\label{same initial letter}
If $w$ and $v$ are minimal words with the same initial letter, then $wv$ is minimal.
\end{corollary}

\begin{proof}
The inequality of Theorem~\ref{minimality} holds for $w$ and for $v$.
Since $(xy)_{wv} = (xy)_w + (xy)_v$, this gives
\begin{align*}
	|(ab)_{wv} - (a\overline{b})_{wv}|
	&= |(ab)_w - (a\overline{b})_w + (ab)_v - (a\overline{b})_v| \\
	&\leq \min((aa)_w, (bb)_w) + \min((aa)_v, (bb)_v) \\
	&\leq \min((aa)_{wv}, (bb)_{wv}). \qedhere
\end{align*}
\end{proof}

The converse is not true in general.
For example, if $w = a$ and $v = abb$, then $wv$ is minimal but $v$ is not.

\section{Root words}\label{root words}

A special case of Corollary~\ref{same initial letter} was observed by Virnig~\cite{Virnig}, namely that if $w$ is a minimal word whose initial letter is $a$, then $aw$ is also minimal.
More generally, since Theorem~\ref{minimality} characterizes minimal words as words having sufficiently many subwords of the form $x^2$, increasing the length of a subword $x^\ell$ in a minimal word produces another minimal word.
This motivates the following definition.

\begin{definition}
Let $w = w_1 \cdots w_\ell \in C_2$ be a nonempty word, and let $1 \leq i \leq \ell$.
The word $w_1 \cdots w_{i-1} w_i w_i w_{i+1} \cdots w_\ell$ obtained by duplicating the letter $w_i$ is a \emph{child} of $w$, and $w$ is one of its \emph{parents}.
Furthermore, we define each letter $x \in L_2$ to be a child of the empty word.
\end{definition}

Thus a child of a minimal word is necessarily minimal.
A word in general does not have a unique parent; for example, the parents of $aabb\overline{aa}bb$ are $abb\overline{aa}bb$, $aab\overline{aa}bb$, $aabb\overline{a}bb$, and $aabb\overline{aa}b$.

Words which cannot be grown as children of minimal words are essentially new minimal words with respect to the operation of duplicating a letter.

\begin{definition}
A \emph{root word} is a minimal word that is not a child of any minimal word.
\end{definition}

\begin{example}
The word $aabb$ is a root word, since it is minimal and neither of its parents $abb$ and $aab$ is minimal.
The words $aba\overline{b}$ and $ab\overline{a}\overline{b}$ are also root words; they are minimal by Theorem~\ref{minimality} and are not children of any minimal word, since in particular they have no subword $xx$.
\end{example}

Root words are the words for which Theorem~\ref{minimality} holds for equality.

\begin{theorem}\label{root word}
Let $w \in C_2$.
The following are equivalent.
\begin{enumerate}
\item
$w$ is a root word.
\item
$|(ab)_w - (a\overline{b})_w| = (aa)_w = (bb)_w$.
\item
$|(yx)_w - (y\overline{x})_w| = (yy)_w = (xx)_w$ for some $x, y \in L_2$ with $x \notin \{y,\overline{y}\}$.
\end{enumerate}
\end{theorem}

\begin{proof}
By Theorem~\ref{minimality}, we have $|(ab)_w - (a\overline{b})_w| \leq (aa)_w$ and $|(ab)_w - (a\overline{b})_w| \leq (bb)_w$.
A minimal word $w$ is a root word if and only if replacing any $xx$ by $x$ in $w$ causes the word to lose minimality.
Shortening such a subword corresponds to decrementing $(aa)_w$ or $(bb)_w$ (or more generally $(yy)_w$ or $(xx)_w$), so therefore $w$ is a root word precisely when both inequalities hold for equality.
\end{proof}

From this characterization we can derive several properties of root words.

\begin{corollary}\label{no inverses}
Let $w$ be a word on the alphabet $\{a, b\}$.
Then $w$ is a root word if and only if $[ab]_w = [aa]_w = [bb]_w$.
\end{corollary}

The analogue of Corollary~\ref{same initial letter} does not hold for root words in general.
For example, $w = aabb$ and $v = aaba\overline{b}a\overline{bb}$ are root words, but $wv$ is not a root word.
However, we do have the following.

\begin{corollary}
Let $n \geq 1$ and $w \in C_2$.
Then $w$ is a root word if and only if $w^n$ is a root word.
\end{corollary}

\begin{proof}
Multiplying $w$ by itself preserves equality in Theorem~\ref{root word}, since all terms scale by the same factor.
Likewise, taking the $n$th root of $w^n$ preserves equality by dividing all terms by $n$.
\end{proof}

\begin{corollary}
If $w$ is a root word, then $(a)_w = (b)_w = |w|/2$.
\end{corollary}

\begin{proof}
The only length-$2$ subwords with unequal generator weights are $aa$, $\overline{aa}$, $bb$, and $\overline{bb}$, but $(aa)_w = (bb)_w$.
\end{proof}

\begin{theorem}\label{divisibility by four}
If $w$ is a root word, then $|w|$ is divisible by $4$.
\end{theorem}

\begin{proof}
We have $|w| = (aa)_w + (bb)_w + (ab)_w + (ba)_w + (a\overline{b})_w + (\overline{b}a)_w$ because these subwords and their inverses constitute the set of length-$2$ subwords.
By Theorem~\ref{practicality} and Theorem~\ref{root word} this simplifies to
\begin{align*}
	|w| &= 2 (aa)_w + 2 (ab)_w + 2 (a\overline{b})_w \\
	&= 2 |(ab)_w - (a\overline{b})_w| + 2 (ab)_w + 2 (a\overline{b})_w.
\end{align*}
If $(ab)_w \geq (a\overline{b})_w$ then $|w| = 4(ab)_w$, and if $(ab)_w < (a\overline{b})_w$ then $|w| = 4 (a\overline{b})_w$.
\end{proof}

For $w$ nonempty, let $\lambda(w) = \max \{ \ell \geq 1 : \text{$(x^\ell)_w \neq 0$ for some $x \in L_2$} \}$ be the length of the longest subword of the form $x^\ell$.
For example, $\lambda(aa\overline{bb}\overline{a}ba\overline{b}a) = 3$.

\begin{theorem}
If $w$ is a nonempty root word, then $\lambda(w) \leq \frac{|w|}{4} + 1$.
\end{theorem}

\begin{proof}
We have $(ab)_w + (a\overline{b})_w \geq |(ab)_w - (a\overline{b})_w| = (aa)_w$.
From the proof of Theorem~\ref{divisibility by four} we therefore obtain $|w| = 2 \left((aa)_w + (ab)_w + (a\overline{b})_w\right) \geq 4 (aa)_w$, so $\lambda(w) \leq \max((aa)_w, (bb)_w) + 1 = (aa)_w + 1 \leq \frac{|w|}{4} + 1$.
\end{proof}

Furthermore, for $n \geq 1$ there exists a root word of length $4n$ that achieves $\lambda(w) = \frac{|w|}{4} + 1$, namely $a^{n+1}(ba)^{n-1}b^{n+1}$, which can be shown to be a root word by Corollary~\ref{no inverses}.

We next prove several lemmas, with the goal of showing in Theorem~\ref{root classes} that a minimal word that is equivalent to a root word is also a root word.
The first lemma describes the effect of a one-letter automorphism on a word.

\begin{lemma}\label{cyclic word}
Let $w \in F_2$, and let $\phi = (\{y\},x)$ with $x \notin \{y,\overline{y}\}$.
Let $v = \phi(w)$; then
\begin{align*}
	(yy)_v &= (y\overline{x}y)_w \\
	(yx)_v &= (yx)_w + (yy)_w \\
	(y\overline{x})_v &= (y\overline{x})_w - (y\overline{x}y)_w \\
	(xx)_v &= (yx)_w - (yx\overline{y})_w + (xx)_w - (y\overline{xx})_w.
\end{align*}
Also, $(yy)_v -(xx)_v = (y\overline{x})_w - (yx)_w - (xx)_w$.
\end{lemma}

\begin{proof}
The subword $(yy)_w$ only appears in $v$ as a result of cancellations in $(y \overline{x} y)_w$ in $w$, which yields the first equality.
The subword $(yx)_w$ remains fixed under $\phi$, and also arises in $v$ from $(yy)_w$ in $w$, which yields the second equality.
Next, $(y \overline{x})_w$ remains fixed under $\phi$, unless it appears in $(y \overline{x} y)_w$, which yields the third equality.
Next, $(xx)_w$ remains fixed under $\phi$ unless it appears in $(y\overline{xx})_w$, and also arises in $v$ from $(yx)_w$, unless this $(yx)_w$ appears in $(y \overline{xy})_w$ in $w$, which yields the fourth equality.
For the final equality, subtract the first two equalities and use the facts that $(yx\overline{y})_w = (y\overline{xy})_w$ (since $(yx\overline{y})^{-1} = y\overline{xy}$) and that $(y\overline{x})_w = (y\overline{x}y)_w + (y\overline{xy})_w + (y\overline{xx})_w$.
\end{proof}

\begin{definition}
An automorphism $\phi$ is \emph{level} on a word $w \in F_2$ if $|\phi(w)| = |w|$.
\end{definition}

For example, each Type I automorphism (permutation) is level on each word in $F_2$.
The following lemma characterizes level one-letter automorphisms on a word.

\begin{lemma}\label{level automorphism}
Let $w \in F_2$ and $x \notin \{y,\overline{y}\}$.
Then the automorphism $(\{y\},x)$ is level on $w$ if and only if $(y\overline{x})_w = (yx)_w + (yy)_w$.
\end{lemma}

\begin{proof}
As in the proof of Lemma~\ref{Rapaport}, the automorphism $(\{y\},x)$ causes cancellations in $w$ only in the subwords $y\overline{x}$ and $x\overline{y}$.
The total number of cancellations is therefore $(y\overline{x})_w$.
Similarly, $(\{y\},x)$ causes additions to $w$ in the subwords $yx$, $\overline{xy}$, $yy$, and $\overline{yy}$, totaling $(yx)_w + (yy)_w$.
We have that $(\{y\}, x)$ is level on $w$ if and only if the number of additions it causes is equal to the number of cancellations it causes, so $(y\overline{x})_w = (yx)_w + (yy)_w$.
\end{proof}

\begin{lemma}
If $w$ is a root word and $\phi$ is a one-letter automorphism which is level on $w$, then $\phi(w)$ is a root word.
\end{lemma}

\begin{proof}
Let $\phi = (\{y\}, x)$ and $v = \phi(w)$.
By Lemma~\ref{cyclic word}, $(yy)_v - (xx)_v = (y\overline{x})_w - (yx)_w - (xx)_w$.
Since $w$ is a root word, $(yy)_w = (xx)_w$ from Theorem~\ref{root word}; and $(y\overline{x})_w = (yx)_w + (yy)_w$ from Lemma~\ref{level automorphism}, so $(y\overline{x})_w - (yx)_w - (xx)_w = 0$, which means that $(yy)_v = (xx)_v$.
By Lemma~\ref{cyclic word}, $|(yx)_v - (y\overline{x})_v| = |(yx)_w + (yy)_w - (y\overline{x})_w - (y\overline{x}y)_w| = |{-}(y\overline{x}y)_w| = (y\overline{x}y)_w = (yy)_v$.
Therefore, by Theorem~\ref{root word}, $v$ is a root word.
\end{proof}

\begin{theorem}\label{root classes}
If $w$ is a root word, $w \sim v$, and $|w| = |v|$, then $v$ is a root word.
\end{theorem}

\begin{proof}
By Whitehead's theorem, $v$ is the image of $w$ under a sequence of one-letter automorphisms, inner automorphisms, and permutations.
Since $|w| = |v|$, each of these automorphisms is level.
By the previous lemma, the image of a root word under a level one-letter automorphism is a root word.
By Theorem~\ref{root word}, the image of a root word under a level inner automorphism or a permutation is also a root word.
Therefore $v$ is a root word.
\end{proof}

In light of this theorem, we refer to an equivalence class containing a root word as a \emph{root class}, since all minimal words in it are root words.
The authors have implemented the algorithm suggested by Whitehead's theorem and have computed all equivalence classes of words in $F_2$ containing a word of length $\leq 20$.
This computation required several months to carry out.
The number of equivalence classes whose minimal words have length $n = 0, 1, 2, \dots$ is
\begin{multline*}
	1, 1, 1, 1, 3, 4, 10, 16, 43, 101, 340, 911, 2544, 7224, 22616, \\
	65376, 187754, 545743, 1653966, 4832057, 14027794, \dots.
\end{multline*}
The number of root classes of each minimal length is
\[
	1, 0, 0, 0, 2, 0, 0, 0, 13, 0, 0, 0, 304, 0, 0, 0, 11395, 0, 0, 0, 478444, \dots.
\]
The explicit equivalence class data is available from the second author's web site, as is the code used to compute it.

We observed above that $aba\overline{b}$ and $ab\overline{a}\overline{b}$ are root words.
We say that a word $w$ is \emph{alternating} if $(xx)_w = 0$ for all $x \in L_2$.
We conclude by showing that alternating root words are characterized by being level under the one-letter automorphisms.
An example of such a word is $(a b \overline{a} \overline{b})^n$ for any $n \geq 0$, which is a root word by Theorem~\ref{root word}.

\begin{theorem}\label{alternating minimal word}
Let $w \in F_2$.
The following are equivalent.
\begin{enumerate}
\item
$w$ is an alternating minimal word.
\item
$w$ is an alternating root word.
\item
The four one-letter automorphisms of Lemma~\ref{one-letter automorphisms} are level on $w$.
\end{enumerate}
\end{theorem}

\begin{proof}
Assume (1).
If $w$ is an alternating minimal word, then $(aa)_w = (bb)_w = 0$, so $(ab)_w = (a\overline{b})_w$ by Theorem~\ref{minimality}.
Therefore, by Theorem~\ref{root word}, $w$ is a root word, so we have (2).

Assume (2) and let $\phi = (\{x\},y)$.
Because $(aa)_w = (bb)_w = 0$ and $(ab)_w = (a\overline{b})_w$, we have $(xy)_w = (x\overline{y})_w$ and $(xx)_w = 0$ for all $x,y \in L_2$, $x \notin \{y,\overline{y}\}$.
Therefore the number of cancellations caused by $\phi$ is equal to the number of additions, and the length of $w$ does not change.
Thus we have (3).

Let all one-letter automorphisms be level on $w$ in accordance with (3).
This implies $(ab)_w - (a\overline{b})_w = (aa)_w$, $(a\overline{b})_w - (ab)_w = (aa)_w$, $(ab)_w - (a\overline{b})_w = (bb)_w$, and $(a\overline{b})_w - (ab)_w = (bb)_w$ by Lemma~\ref{level automorphism}, so $(aa)_w = (bb)_w = 0$ and $(ab)_w = (a\overline{b})_w$.
Therefore $w$ is minimal and alternating, so (1) holds.
\end{proof}

\end{document}